\documentclass{amsart}
\usepackage{amsmath,amsthm,amssymb,graphicx,mathtools,mathrsfs}
\usepackage{cancel,mathrsfs,hyperref,marginnote,bm}
\usepackage{color}

\makeatletter
\@namedef{subjclassname@2020}{2020 Mathematics Subject Classification}
\makeatother

\numberwithin{equation}{section}

\newcommand{\Z}{\mathbb{Z}}
\newcommand{\Q}{\mathbb{Q}}

\newcommand{\F}{\mathbb{F}}

\newcommand{\cB}{\mathcal{B}}

\newcommand{\cT}{\mathcal{T}}

\newcommand{\bSig}{\mathbf\Sigma}
\newcommand{\Sb}{\bSig_\bullet}
\newcommand{\Ib}{\cI_\bullet}
\newcommand{\sbb}{\sigma_\bullet}
\renewcommand{\bar}{\overline}

\newcommand{\sL}{\mathscr{L}}

\newcommand{\cC}{\mathcal{C}}
\newcommand{\cI}{\mathcal{I}}

\newcommand{\bi}{\mathbf{i}}

\newcommand{\bt}{\mathbf{t}}

\newcommand{\pceil}[1]{\lceil #1 \rceil}

\DeclareMathOperator{\Gal}{Gal}

\DeclareMathOperator{\Tr}{Tr}
\DeclareMathOperator{\id}{id}

\DeclareMathOperator{\Mat}{Mat}

\DeclareMathOperator{\sgn}{sgn}

\DeclareMathOperator{\NP}{NP}
\DeclareMathOperator{\Supp}{Supp}

\newcommand{\aff}{\mathrm{aff}}

\theoremstyle{plain}
\newtheorem{theorem}{Theorem}[section]
\newtheorem{corollary}[theorem]{Corollary}
\newtheorem{lemma}[theorem]{Lemma}
\newtheorem{keylemma}[theorem]{Key Lemma}
\newtheorem{proposition}[theorem]{Proposition}

\theoremstyle{definition}
\newtheorem{definition}[theorem]{Definition}
\newtheorem{example}[theorem]{Example}
\newtheorem{remark}[theorem]{Remark}
\newtheorem*{remark*}{Remark}
\newtheorem*{example*}{Example}
\newtheorem{acknowledgments}{Acknowledgments}

\title[Curves With Controlled First Slope]
{Construction of Curves with
a Controlled First Slope using $p$-Symmetric Numbers}
\author{Robert Moore, Hui June Zhu}
\date{\today}

\address{
Robert Moore,
Department of mathematics,
State University of New York at Buffalo,
Buffalo, NY 14260. 
The United States.
}
\email{rcmoore@buffalo.edu}

\address{
Hui June Zhu,
Department of mathematics,
State University of New York at Buffalo,
Buffalo, NY 14260. 
The United States.
}
\email{hjzhu@math.buffalo.edu}



\begin{document}

\begin{abstract}
This paper establishes a constructive link between the first slope of the
Artin-Schreier curve $X_f: y^p-y=f(x)$ along with its length and 
the $p$-adic weight of the support of $f(x)$.
If the maximal $p$-adic weight element $\nu$ in $\Supp(f)$ is unique,
we show that the first slope's lower bound of $1/s_p(\nu)$
is achieved if and only if 
$\nu$ satisfies an elementary combinatorial $p$-adic condition, which we define as
{\em $p$-symmetry}. In this case, we give the length of
the first slope segment explicitly.
As an application, we construct explicit families of 
curves in every characteristic $p$ with 
first slope equal to $1/n$ for every $n\ge 2$. 
\end{abstract}

\maketitle

\section{Introduction}

A curve in this paper is a smooth projective and 
geometrically integral algebraic variety of dimension $1$.
Let $p$ be a rational prime. 
Let $q$ be a $p$-power and let $v_q(\alpha)$ 
be the additive $q$-adic valuation 
of $\alpha\in\bar\Q_p$ with $v_q(q)=1$.  
Let $X$ be a curve over $\F_q$ of genus $g$.
Write $Z(X/\F_q,s)$ for the zeta function of $X$. 
Then by Weil (\cite{We48}), $Z(X/\F_q,s)(1-s)(1-qs)=\prod_{i=1}^{2g}(1-\alpha_i s)$ 
where $\alpha_i$ are Weil $q$-numbers.
We shall order these reciprocal zeros such that
$v_q(\alpha_1)\le v_q(\alpha_2)\le \dots \le v_q(\alpha_{2g})$.
It is known that a generic curve of genus $g$ over $\F_q$ 
is {\em ordinary}, where
$v_q(\alpha_1)=\dots=v_q(\alpha_g)=0$ and 
$v_q(\alpha_{g+1})=\dots=v_q(\alpha_{2g})=1$.
On the other hand, a curve $X$ is {\em supersingular}
if and only if  $v_q(\alpha_1)=\dots = v_q(\alpha_{2g})=\frac{1}{2}$. 
The Jacobian variety of a supersingular curve 
is isogenous over $\bar\F_q$ to a product of $g$ supersingular elliptic curves.
The {\bf Newton polygon} of $X$ is the piecewise linear function 
on the real plane on the interval $[0,2g]$ with slopes 
$v_q(\alpha_1),\ldots,v_q(\alpha_{2g})$, each of multiplicity $1$,
denoted by $\NP(X)$. 
Each Newton polygon is 
a lower convex hull with only integral vertices.
Every possible Newton polygon of $X$ lies between the supersingular 
and the ordinary ones. 
The first slope of the Newton polygon of $X$ (called simply the {\bf first
slope} of $X$) has been studied by
Ax \cite{Ax64}, Katz \cite{Ka71}, Moreno-Moreno \cite{MM95}, Adolphson-Sperber 
\cite{AS87}, and Wan \cite{Wan95}.    

Much is known about  the Newton polygon stratification of the moduli space of 
principally polarized abelian varieties (see \cite{Oo91b}). 
By contrast, it is still open in many cases 
whether a given Newton polygon can be realized by the Jacobian of a curve; 
see \cite{Oo91a}.
There has been a focus on the supersingular strata in the past decades:
Li and Oort \cite{LO98} proved that the supersingular locus has dimension 
$[\frac{g^2}{4}]$ in the moduli of abelian varieties. 
On the other hand, the Torelli locus (which parameterizes Jacobians of curves) 
is of dimension  $3g-3$ (for $g>1$).
It is not clear whether the supersingular locus and the Torelli locus 
have a non-empty intersection.
van der Geer and van der Vlugt (\cite{GV92}, \cite{GV95})
explicitly constructed supersingular curves of 
every genus $g$ when $p=2$, built exclusively upon
Artin-Schreier curves.  
For more recent progress on the existence of supersingular curves 
over fields of arbitrary odd characteristic, see \cite{KHS20} for genus 4 for example.

By comparison, the Torelli locus 
of non-supersingular and non-ordinary curves remains poorly understood, and 
very little is known or even conjectured.
Zarhin (see \cite{Zar04}) identified a sufficient criterion
for non-supersingularity of hyperelliptic curves in characteristic $p>2$,
but it remains an open question whether,
for instance,
one can construct a family of curves 
with a prescribed non-supersingular, non-ordinary Newton polygon.

The main object of study in this paper will be Artin-Schreier curves
\[
    X_f : y^p - y = f
\]
for some $f = \sum_{i} a_i x^i\in\F_q[x]$
of degree $d\ge 3$. 
For the rest of the paper, we assume $p\nmid i$ for all $i\in \Supp(f)$.
Notice that for an arbitrary polynomial $h$ one can 
replace it by $h-h(0)$ without altering the Newton polygon
since $X_h$ and $X_{h-h(0)}$ share the same one. 
If there are nonzero multiples of $p$ in $\Supp(h)$, 
one can find a polynomial $f$ satisfying $p\nmid i$ for all $i\in \Supp(f)$, 
where $X_f$ is isomorphic to $X_h$, hence $X_f$ and $X_h$ have the 
same Newton polygon.

For a nonnegative integer $N$, let $s_p(N)$ denote the
$p$-adic weight of $N$, i.e., the sum of $N$'s $p$-adic digits.
It is well-known that the first slope of
$X_f$ is bounded below by $\frac{1}{\max_{i \in \Supp(f)}s_p(i)}$
(see a proof in \cite[Theorem 4.1]{Wan95}). 
In \cite{MSCK04}, the authors describe an open condition on the coefficients
of $f$ which, when achieved, guarantees that the first slope
of $X_f$ is equal to their lower bound.
In \cite{Bl12} Blache introduces the $p$-density of $f$, 
which can be used to compute the first slope.    
However, neither the $p$-density nor this open condition is easy to compute
for a family of $X_f$ with fixed support set for $f$. 

In this paper, we identify a class of sets $S$
where the lower bound of $\frac{1}{\max_{i \in \Supp(f)}s_p(i)}$
is achieved for every $f(x)$ with $\Supp(f) = S$.
In fact, if $\Supp(f)$ contains a unique element $\nu$ of maximal $p$-adic
weight, then we are able to completely characterize when the lower bound of
$\frac{1}{\max_{i \in \Supp(f)}s_p(i)}$ is achieved.
We do this
by establishing a sufficient condition on elements in $\Supp(f)$,
which we define as $p$-symmetry.
A positive integer $\nu$ coprime to $p$ is 
{\bf $p$-symmetric} if 
there exists a $3$-tuple of positive integers $(k,\ell,w)$
with $\ell < p^k$ such that 
$\nu w= (p^k-1)\ell$ where the multiplication $\nu w$ in $p$-adic digit form is
carry-free (see Section \ref{S:p-weight}).
For example, $5$-symmetric numbers of precisely $3$ digits are
(exhaustive in $5$-adic form):
\begin{align*}
&(111)_5, (222)_5,
(444)_5, (101)_5, (202)_5, (404)_5,
(121)_5, (242)_5, (143)_5, (341)_5, (301)_5, \\
&(103)_5, (112)_5,  (211)_5,
\end{align*}
which counts $14\%$ of all 3-digit numbers
coprime to $5$.
In decimal representation and in increasing order, they are 
$26,28,31,32,36,48,52,56,62,72,76,96,104,124.$
See Section \ref{S:4.1}
for an exploration of $p$-symmetric numbers.

\begin{theorem}
\label{T:main-thm-tight}
Let $f=\sum_{1\le i\le d, p\nmid i}a_i x^i \in\F_q[x]$.        
Then the first slope of the curve
    $X_f : y^p - y = f$ is $\ge \frac{1}{\max_{i\in\Supp(f)} s_p(i)}$.
\begin{enumerate}
\item     Suppose $\Supp(f)$ contains a unique element $\nu$ such that
    \[
        s_p(\nu)=\max_{i\in\Supp(f)}s_p(i).
    \]
    Then the first slope is equal to $\frac{1}{s_p(\nu)}$ 
    if and only if $\nu$ is a $p$-symmetric number.
\item If $\nu=p^k-1$ is the unique maximal $p$-adic weight
    element in $\Supp(f)$, then 
the first slope is equal to $\frac{1}{k(p-1)}$ and has multiplicity $k(p-1)$.
\end{enumerate}
\end{theorem}
To highlight the effectiveness of Theorem \ref{T:main-thm-tight},
we will examine some of its applications.

Firstly, suppose $f$ is of degree $p^k-1$ for some $k\ge 1$. 
Then $p^k-1$ must be the unique maximal $p$-adic weight element in $\Supp(f)$. 
Moreover, $p^k - 1$ is $p$-symmetric, with $(k,\ell,w) = (k,1,1)$.
By Theorem \ref{T:main-thm-tight},
the first slope of $X_f$ is equal to $\frac{1}{k(p-1)}$
and has multiplicity $k(p-1)$.
Therefore, we see that the first slope of $X_f$ and its multiplicity 
are completely determined when $f$ is of degree $p^k - 1$.
We actually give an estimate of the multiplicity for many general cases in 
Theorem \ref{T:main!}.

Our next application of Theorem \ref{T:main-thm-tight} is 
for the case when $p > d$, where $d = \deg(f)$.
In this case, $d$ is the element in $\Supp(f)$
of maximal $p$-adic weight.
If $p \equiv 1 \bmod{d}$,
then $d$ is clearly $p$-symmetric.
In fact, since $d<p$, we can show 
$d$ is $p$-symmetric if and only if $p\equiv 1\bmod d$
(see a proof of this in Example \ref{ex:p-symmetry}(4)). 
Theorem \ref{T:main-thm-tight} then implies that the first slope is
$1/d$ precisely when this occurs.
This is already known (see \cite{AS89}\cite{Wan93}\cite{SZ03}),
however our result provides a more satisfying
explanation for this fact.

Theorem \ref{T:main-thm-tight} actually allows for a strong generalization
of this result.
One immediate extension stems from the fact that $d$ need not be the
element in $\Supp(f)$ of maximal $p$-adic weight.
If $\nu \in \Supp(f)$ is the unique element of maximal $p$-adic weight
and $\nu < p$, then the first slope of $X_f$ is equal to $1/\nu$
if and only if $p \equiv 1 \bmod{\nu}$.
This is true
regardless of the degree of $f$.
This generalization extends further in the following corollary:

\begin{corollary}
    \label{C:spd-family-1}
   Suppose $\Supp(f)$ has a unique maximal $p$-adic weight element $\nu$. 
If $p\equiv 1\bmod s_p(\nu)$, 
then the first slope of $X_f$ is equal to $1/s_p(\nu)$.
\end{corollary}

Proving this corollary requires showing that
$p \equiv 1 \bmod{s_p(\nu)}$ implies that $\nu$
is $p$-symmetric.
This is done in Section \ref{S:proof},
Corollary \ref{C:spd-family}.

A third application of Theorem \ref{T:main-thm-tight} involves building on the
supersingular curves due to van der Geer and van der Vlugt
(see \cite{GV92}), whose construction used a coding theory--inspired method.
The van der Geer--van der Vlugt supersingular curves are fiber products of
Artin-Schreier curves 
of the form $X : y^p - y = x\sum_{i=0}^k a_i x^{p^i}$.

\begin{corollary}
    \label{C:ss-van-der-Geer}
    If $c_\nu\ne 0$ for some $p$-symmetric number $\nu>1$
    and $\nu \ne p^i+1$ for any $i$, then any curve of the form
$X: y^p-y = c_\nu x^\nu + x\sum_{i=0}^{k}a_i x^{p^i}$ has first slope equal to
$1/s_p(\nu)$ and is non-supersingular.
\end{corollary}

\begin{proof}
    The $p$-symmetry of $\nu$ implies that it is coprime to $p$,
    and the fact that $\nu > 1$ and $\nu \ne p^i + 1$ for any $i$
    implies that $s_p(\nu) > 2$.
    Since $s_p(p^i + 1) = 2$ for all $i \ge 0$,
    we see that $\Supp(f)$ has $\nu$ as a unique element of maximal $p$-adic weight.
    Since it is also $p$-symmetric, 
    we have by Theorem \ref{T:main-thm-tight}
    that the first slope of $X$ is $1/s_p(\nu)$.
\end{proof}

\begin{remark*} 
In a special case of Theorem \ref{T:main-thm-tight}(2) 
where the polynomial $f$ is of degree $p^k-1$ (in other words, 
$p^k-1$ is the maximal integer in $\Supp(f)$), 
the first slope of $X_f$ and its multiplicity are established by Section 3.1 of \cite{Bl21}.
In this special instance, our Proposition \ref{P:pk-1-case} singles out precisely the set $\{1,p,\ldots,p^{k-1}\}$, which corresponds to Blache's minimal support. 
By comparison, Blache's minimal support records the relevant vertices of minimal solutions, while our minimizer (in Definition \ref{D:minimizer}) records the same vertices together with permutation/cycle data between them. 
\end{remark*}

\vspace{5mm}

\noindent{\bf Organization of the paper:}
We prove or recall some $p$-adic weight properties in Section \ref{S:p-weight}, 
and a key lemma regarding solutions to a $p$-adic change-making 
problem in Section \ref{S:change-making}.
Section \ref{S:symmetric} introduces $p$-symmetric numbers 
and minimizers, and proves a number of important combinatorial lemmas. 
These are applied to the
$p$-adic estimate of the characteristic power series 
of the Dwork operator in Section \ref{S:estimate}.
Our main theorems
and their applications are obtained in Section \ref{S:proof}.
In particular, our main result of
Theorem \ref{T:main-thm-tight}
follows from Theorem \ref{T:main!},
and with it we construct several families of curves
with a given first slope.
Many non-supersingular curves emerge among these families.

\section{\texorpdfstring{$p$}{p}-adic weight under addition and multiplication}
\label{S:p-weight}
The $p$-adic weight of a nonnegative integer $N$,
which we henceforth denote as $s_p(N)$,
gives the sum of the $p$-adic digits of $N$.
That is, if $N = \sum_{i=0}^\ell n_i p^i$, $n_i \in \{0,\dots,p-1\}$,
then
$s_p(N)\coloneqq \sum_{i=0}^\ell n_i$.
Frequently in this paper, for specific $N$, we will write
$N=(n_\ell\; \cdots\; n_1 \; n_0)_p$
for the {\bf $p$-adic (digit) form} of $N$,
and we call $n_i$ the {\bf $i$-th $p$-adic digit of $N$}.

Addition of two nonnegative integers $N$ and $M$ can be done
$p$-adically, where the arithmetic addition algorithm is done
using the $p$-adic digit forms of $N$ and $M$.
For example, the $2$-adic addition of $3+10$ equals
$(11)_2+(1010)_2=(1101)_2$.
Notice that there is 
a {\bf ($p$-adic) carry} in the $2$-nd digit. 
Given any nonnegative integers $N_1,\ldots,N_\ell$, 
we say the summation $\sum_{i=1}^{\ell}N_i$ 
is $p$-adically {\bf carry-free} if their $p$-adic addition
has no carry at each digit. 
We simply say `carry-free' when $p$ is clear from its context.

Multiplication of two nonnegative integers $N$ and $M$
can also be done $p$-adically,
using standard arithmetic rules in base $p$.
In the rare case where the $p$-adic multiplication 
of $N$ and $M$ has no carrying, the $k$-th $p$-adic digit
of $NM$ is equal to the convolution $n_0m_k + n_1m_{k-1} + \dots + n_km_0$.
Thus, we can say that multiplication is $p$-adically {\bf carry-free}
if $\sum_{i+j=k}m_i n_j\le p-1$ for every $k\ge 0$.

The notion of `$p$-adically carry-free' is given a more precise treatment
in the following proposition --- which we will use
to swiftly replace the notion of ``carry-free addition or multiplication''
by a more convenient characterization in terms of $p$-adic weights. 

\begin{proposition}[Triangle inequality]
  \label{P:2.1}
    Let $N,M \in \Z_{\ge 0}$. Then 
    $$s_p(N)+ s_p(M) \ge s_p(N+M), \quad s_p(N)s_p(M) \ge s_p(NM).
$$
The first equality holds if and only if the 
    $p$-adic addition $N+M$ is carry-free;
    the second equality holds if and only if the $p$-adic product 
    $NM$ is carry-free.
\end{proposition}
\begin{proof}
The first statement is classical; see \cite[Proposition 2.2]{HLS11} for example.
The remaining assertions follow directly from the base-$p$ carrying algorithm.
\end{proof}

\begin{proposition}
    \label{lem:sp_qi-j}
Let $a\in\Z_{\ge 1}$.
    Let $i,j$ be positive integers with $p^ai - j \ge 0$.  Then
    \[
        s_p(p^ai - j) \ge a(p-1) + s_p(i-1) - s_p(j-1).
    \]
    Equality holds if $j\le p^a$.
\end{proposition}
\begin{proof}
By Proposition \ref{P:2.1} above, 
we have
\begin{align*}
    s_p(p^ai-j)+s_p(j-1) &\ge s_p(p^ai-1)\\
                         &=s_p(p^a(i-1)+(p^a-1))\\
                         &=s_p(i-1)+s_p(p^a-1)=
    s_p(i-1) + a(p-1).
\end{align*}
If $j\le p^a$, then the summation of $(p^ai-j)+(j-1)$ is carry-free, hence
the equality holds by Proposition \ref{P:2.1}.
\end{proof}
We take a moment to highlight a particular case of Proposition
\ref{lem:sp_qi-j}: when $i=j$ and $i \le p^a$, we have
$s_p(i(p^a-1)) = a(p-1)$.

When we wish to emphasize when the product $NM$ is carry-free,
we will often typeset it as $N*M$.
For example,
$(110011)_2 * (101)_2$ is carry-free, but $(110011)_2\times (11)_2$
and $(242)_5\times (12)_5$ are not.

\section{A change-making problem and its 
solution value}
\label{S:change-making}
This section is a technical preparation
surrounding a particular
$p$-adic change-making problem.
In general, a ``change-making problem'' involves
finding the minimal number of ``coins'' in a given coin set 
needed to achieve
a particular fixed value.
More explicitly, given a set $\cC = \{c_1, \dots, c_m\}$ of positive integers
and a fixed integer $N$, the change-making problem involves finding
a tuple $(t_1, \dots, t_m)$ of nonnegative integers such that
$t_1 c_1 + \dots + t_m c_m = N$, and that minimizes $t_1 + \dots + t_m$.
See \cite{CLRS} for a broader analysis of change-making problems.

Let $\bf{i}$ be a finite set of positive integers coprime to $p$.
For any integer $a\ge 1$,  consider the index set of pairs
$$
    I \coloneqq \bi \times \{0, 1, \dots, a-1\}.
$$
We will fix a \textbf{system of coins} or \textbf{coin set}
\begin{equation}\label{E:coin_set}
\cC = \cC(\bi,a):=\left\{ip^j\mid i\in \bi, j=0,1,\ldots, a-1\right\}.
\end{equation}
Each integer $c_{ij} \coloneqq ip^j$ in $\mathcal{C}$ is called a \textbf{coin}.
Since every $i\in \bi$ is coprime to $p$, the indexing map $(i,j) \mapsto c_{ij}$
is a bijection between $I$ and $\cC$.
For every integer $N \ge 0$, let
\begin{equation}
\label{E:I(N)}
    \cT(N) = \left\{(t_{ij})_{(i,j)\in I} \in \Z_{\ge 0}^{I}
    \mid \sum_{(i,j)\in I} t_{ij}(ip^j) = N\right\}.
\end{equation}
Each element in $\cT(N)$ is called a \textbf{representation of $N$ (under $\cC$)}.

If a representation $\bt = (t_{ij})\in \cT(N)$ has the minimal $|\bt|\coloneqq\sum_{(i,j)\in I}t_{ij}$, among all elements in $\cT(N)$, then it is 
called a \textbf{solution to the $p$-adic change-making problem for $N$ (under
$\cC$)}.
Then we call 
$$M_{\cC}(N) \coloneqq |\mathbf{t}|=\sum_{i\in\bi}\sum_{j=0}^{a-1} t_{ij}$$
the \textbf{solution value} to the $p$-adic change-making problem. 
In other words, $M_\cC(N)$ is the minimal number of 
coins one needs to sum up to $N$.

The following lemma establishes a lower bound for solution values to the
change-making problem, and identifies certain cases where that
lower bound is both tight and achieved at a unique solution.
\begin{keylemma}
    \label{L:key}   
    Let $\cC=\{c_{ij}=ip^j|i\in \bi,  j=0,1,\ldots, a-1\}$
    where $\bi$ is a finite set consisting of positive integers
    coprime to $p$ and $\nu\in\bi$ such that $s_p(\nu)=\max_{i\in\bi}s_p(i)$.
    Then:
\begin{enumerate}    
\item For every nonnegative integer $N\ge 0$,
    \begin{equation*}
        M_\cC(N) \ge \frac{s_p(N)}{s_p(\nu)}. 
    \end{equation*}
\item If there is a carry-free factorization $N = \nu * w$
    for some $w < p^a$, then $M_\cC(N) = \frac{s_p(N)}{s_p(\nu)}$.
    If $\nu$ is the unique maximal $p$-adic weight element in $\bi$,
    then the converse is also true.
    In this case, there is a unique solution $(t_{ij})$
    to the change-making problem for $N$ under $\cC$, with 
    $w = \sum_{j=0}^{a-1} t_{\nu j}p^j$ and $t_{ij} = 0$ for all $i\ne \nu$.
    This expression gives the $p$-adic form of $w$; i.e., $t_{\nu j} \in
    \{0,\dots, p-1\}$ for all $j$.
\end{enumerate}
\end{keylemma}      
      
\begin{proof}      
(1) Let $\bt = (t_{ij})$ 
be a solution to the change-making problem for $N$ under $\cC$.
Then $N=\sum_{(i,j)\in I}t_{ij}c_{ij}$.
Note that $s_p(c_{ij}) = s_p(ip^j) = s_p(i) \le s_p(\nu)$
by hypothesis. 
Now by Proposition \ref{P:2.1},
we have 
\begin{equation}\label{E:chain}
s_p(N)\stackrel{(*)}{\le} \sum_{(i,j)}t_{ij}s_p(ip^j)=\sum_{(i,j)}t_{ij}s_p(i)
\stackrel{(**)}{\le} s_p(\nu)\sum_{(i,j)}t_{ij}=s_p(\nu) M_\cC(N).
\end{equation}
This immediately proves the desired inequality   
$
M_\cC(N)\ge\frac{s_p(N)}{s_p(\nu)}.
$
      
(2) 
Suppose $N=\nu w$ is carry-free. By Proposition \ref{P:2.1} $s_p(N)=s_p(\nu)s_p(w)$.
Write $w=\sum_{i=0}^{a-1}t_ip^i$ in $p$-adic form.
Then $N = \nu w = \sum_{i=0}^{a-1} t_i \nu p^i$.
But $\nu p^i \in \cC$, so $M_\cC(N) \le \sum_{i=0}^{a-1} t_i = s_p(w) =
\frac{s_p(N)}{s_p(\nu)}$.  Combined with the inequality of part (1),
the equality must hold.

Now suppose that $\nu$ is the unique maximal $p$-adic weight element in $\bi$.
Let $\bt=(t_{ij})$ be a solution to the change-making problem of $N$, 
$M_\cC(N)=\sum_{(i,j)}t_{ij}$ and 
$N=\sum_{(i,j)}t_{ij}i p^j$.
By hypothesis, we suppose $s_p(N)=s_p(\nu)M_\cC(N)$. Then all the equalities of \eqref{E:chain} hold. 
The ($**$) equality holds if and only if $s_p(i)=s_p(\nu)$ for every $t_{ij}\ne 0$ and hence 
$t_{ij}=0$ for all $i\ne \nu$.
The ($*$) equality holds if and only if the sum is carry-free. Combined with the above, 
the sum $\sum_{j=0}^{a-1} t_{\nu j} \nu p^j$ is carry-free, and this immediately implies
$t_{\nu j}\in\{0,1,\ldots,p-1\}$.
Now  $w\coloneqq\sum_{j=0}^{a-1}t_{\nu j}p^j$ is in its $p$-adic form.
Hence $w<p^a$. Notice that $\nu w = \nu\sum_{j=0}^{a-1} t_{\nu,j}p^j=N$. 
On the other hand, $s_p(\nu)s_p(w)=s_p(\nu)\sum_{(i,j)}t_{ij}=s_p(\nu) M_\cC(N)=s_p(N)$.
By Proposition \ref{P:2.1}, the product $\nu * w = N$ is carry-free.

In this case, we have $N=\nu * w$ where $w=\sum_{j=0}^{a-1} t_{\nu j}\ p^j$ in $p$-adic form.
Thus the solution $(t_{ij})$ is clearly unique since the $p$-adic form of $w$ is unique.
\end{proof}

We have observed in the Key Lemma that if the minimum 
$M_\cC(N) = \frac{s_p(N)}{s_p(\nu)}$ is achieved, then $\nu\mid N$. This depends on the given coin set $\cC$. 
If the coin set $\cC$ contains  
$\nu,\nu p,\ldots, \nu p^{a-1}$ whose $p$-adic weight 
$s_p(\nu)$ is the unique maximal, then we have the necessary and sufficient condition for the equality.

\section{\texorpdfstring{$p$}{p}-symmetric numbers and minimizers}
\label{S:symmetric}

Key Lemma \ref{L:key} illustrates the utility of carry-free products
when solving change-making problems for coin sets $\cC$ of Section \ref{S:change-making}.
The Dwork-theoretic approach we will employ
in Section \ref{S:estimate} will involve solving change-making problems
for values $N$ of the form $(q-1)\ell$;
therefore we dedicate an entire section here to describing carry-free
products for numbers of this form.

\subsection{\texorpdfstring{$p$}{p}-symmetric numbers}
\label{S:4.1} 

Fix any prime $p$ and a positive integer $n$.
We know the $n$-digit number with the highest $p$-adic weight is
$p^n-1=(p-1,p-1,\ldots,p-1)_p$ with $s_p(p^n-1)=n(p-1)$. 
More generally, consider the numbers $N=(p^k-1)\ell$ where $\ell<p^k$.
Then $s_p(N)=k(p-1)$ by Proposition \ref{lem:sp_qi-j}.
For example, $N=(3311)_5=(5^2-1)\cdot 19$ has $s_5(N)=8=2(5-1)$,
and $N=(11111111)_2=2^8-1$ has $s_2(N)=8=8(2-1)$.
Here, we focus on the numbers $\nu\in\Z_{\ge 1}$ which factor
numbers of the form $N=(p^k-1)\ell$ in a carry-free manner.

\begin{definition}\label{def:p-symmetry}
    An integer $\nu > 1$ coprime to $p$
    is \textbf{$p$-symmetric}
    if there exists a positive integer $w$ such that 
    the product $\nu w$ is carry-free
    and $\nu * w=(p^k - 1)\ell$
    for some positive integers $k$ and $\ell<p^k$.
\end{definition}

\begin{remark}
 The carry-free aspect of $p$-symmetric numbers has a combinatorial
    interpretation to it, similar to the $r$-tiling sequences described
    in \cite{SZ02c}.
    The number $N = (p^k-1)\ell$ may be expressed $p$-adically as 
    $N = (N_r \; N_{r-1} \; \cdots \; N_0)_p$,
    and the number $\nu$ may be expressed $p$-adically as
    $\nu = (\nu_s \; \nu_{s-1} \; \cdots \; \nu_0)_p$.
    When viewed as finite sequences of length $r+1$ and $s+1$
    respectively, $N$ having a $p$-adic carry-free factorization with $\nu$
    amounts to whether $(N_r, \dots, N_0)_p$ is equal to the 
    $p$-adic carry-free sum 
    of some shifts $(\nu_s,\dots,\nu_0,0,\dots,0)_p$ of $(\nu_s, \dots, \nu_0)_p$.
    Families of $p$-symmetric numbers are demonstrated in 
    Example \ref{ex:p-symmetry} below.
\end{remark}

\begin{proposition}
\label{P:determined}
Let $\nu$ be $p$-symmetric  and suppose
$\nu * w = (p^k-1)\ell$ with $1\le \ell<p^k$.
Then $k$ and $\ell$ are uniquely determined by $w$.
\end{proposition}

\begin{proof}
Suppose also $\nu*w=(p^{k'}-1)\ell'$ with $1\le \ell'<p^{k'}$.
Suppose $\ell\neq \ell'$, without loss of generality, and let $k< k'$. Since
$(p^k-1)\ell\equiv (p^{k'}-1)\ell'\bmod p^k$, we can write
$\ell=\ell'+cp^k$ for some $c\ge 1$. Then $\ell\ge p^k$, in contradiction to our hypothesis; hence $\ell=\ell'$.
\end{proof}

\begin{definition}\label{D:symmetric_number}
If $\nu$ is $p$-symmetric, we pick the smallest such $w$ 
such that $\nu * w$ is a carry-free product and $\nu * w = (p^k-1)\ell$
with $\ell<p^k$. Then such $k, \ell$ are unique
by Proposition \ref{P:determined}.
This equation is called the {\bf minimal factorization} for $\nu$.
If $\nu * w=(p^k-1)\ell$ is the minimal factorization
and $w = (w_e \; \dots \; w_0)_p$ with $w_e \ne 0$, then the nonnegative integer $e$
is called the {\bf shift factor} of $\nu$.
\end{definition}

\begin{example}[$p$-symmetric numbers and their properties]
\label{ex:p-symmetry}
    \-
    \begin{enumerate}
    \item 
    Any integer in $\Z_{>1}$ of the form $(p^k - 1)\ell$ is $p$-symmetric 
    if $\ell\in\Z_{\le p^k}$ and $\ell$ is coprime to $p$.
    In particular, $p^k-1$ is always $p$-symmetric. 
    \item As previously stated, the integer $w$ in Definition
    \ref{def:p-symmetry} is not unique.
    If $\nu$ is $p$-symmetric with $\nu * w = (p^k-1)\ell$,
    then for every integer $b \ge 1$,
    we have a carry-free factorization
    $\nu * w' = (p^{kb}-1)\ell$,
    where $w' = w(1 + p^k + \dots + p^{(b-1)k})$.
\item 
If $\nu$ is a coprime-to-$p$ integer in $\Z_{>1}$ with $s_p(\nu)\mid (p-1)$
then one can show $\nu$ has to be $p$-symmetric (see the proof of
Corollary \ref{C:spd-family} for an argument).
For example, if $p=5$ then 
the following are $p$-symmetric: $(101)_5$,
$(202)_5$, $(103)_5$, $(301)_5$, $(112)_5$, $(121)_5$, $(211)_5$.

    \item Any positive integer $1<\nu<p$ is $p$-symmetric if and only if
    $\nu\mid (p-1)$. If $\nu|(p-1)$ 
then $\nu*\frac{p-1}{\nu} = p-1$, hence $\nu$ is clearly $p$-symmetric; 
conversely, if $\nu$ is $p$-symmetric, then $\nu* w=(p^k-1)\ell$ for some $\ell<p^k$.
Write $w=(w_e\cdots w_0)_p$ in standard $p$-adic digit form.
Since $\nu<p$ and $\nu*w$ is carry-free, we have $\nu w_i\le p-1$, 
hence $\nu *w =(\nu w_e, \ldots,\nu w_0)_p$ in the standard $p$-adic digit form.
We may write $\ell=(\ell_{k-1},\ldots,\ell_0)_p$ where $\ell_0>0$ and $0\le \ell_i\le p-1$.
Then 
$$ \;\;\; \; p^k\ell-\ell=(\ell_{k-1},\ldots,\ell_1,\ell_0-1, p-1-\ell_{k-1},\ldots,
p-1-\ell_1,p-\ell_0)_p.$$ 
Considering the standard $p$-adic digit form of the equation 
$\nu*w =(p^k-1)\ell$, we sum up 
the $k$-th and $0$-th digits and get 
$\nu w_k+\nu w_0 = (\ell_0-1)+(p-\ell_0)$.
That is, $\nu(w_k+w_0)=p-1$. Thus we have $\nu\mid (p-1)$.

    \item If $w \mid p-1$, 
    then $\nu=(p^b-1)/w$ is $p$-symmetric for every $b\ge 1$ unless $\nu=1$.
    This follows from the obvious carry-free product 
    $       \nu w = p^b - 1.
    $
Its shift factor is $0$.
   For example, when $p=7$, the following are $7$-symmetric: $(666)_7$,
   $(333)_7$, $(222)_7$, $(111)_7$.

    \item As a dual to (2), 
        if $\nu$ is $p$-symmetric with factorization $\nu * w = (p^k-1)\ell$,
    then 
    $\nu'=\nu(1 + p^k + p^{2k} + \dots + p^{(b-1)k})$
    is $p$-symmetric for every $b\ge 1$.
    Part (5) is the special case $k=1$, $\ell=1$.
\item If $\nu$ is $p$-symmetric, 
    then its $p$-adic digit reversal is also $p$-symmetric.
    In other words, if $\nu=\sum_{i=0}^{s}\nu_i p^i$ is $p$-symmetric,
    so is its digital reverse $\nu':=\sum_{i=0}^{s}\nu_{s-i}p^i$. Indeed, if $\nu*w=(p^k-1)\ell$ is the minimal factorization, one can check that
 $\nu'*w'=(p^k-1)\ell'$, hence $\nu'$ is $p$-symmetric too.
 For example, 
    $(341)_5$ and $(143)_5$, $(301)_5$ and $(103)_5$, $(112)_5$ and $(211)_5$
    are such pairs of $5$-symmetric numbers.

\item 
If $\nu$ is $p$-symmetric with a minimal factorization $\nu* w=(p^k-1)\ell$ and $w>1$, 
then $w$ is also $p$-symmetric.
 
\item For any $n \ge 2$, $m \ge 0$,
    there exists at least one $p$-symmetric number $\nu$
    with $s_p(\nu) = n$ and shift factor $e\le m$.
    Namely, let $\nu = 1 + p^{m+1} + p^{2(m+1)} + \dots + p^{(n-1)(m+1)}$.
    Then $\nu$ is $p$-symmetric with
    $\nu * (p^{m+1}-1) = p^{n(m+1)}-1$.
    \end{enumerate}
\end{example}

\begin{example}[Shift factor of a $p$-symmetric number]
\label{ex:shift_factor}
\-
\begin{enumerate} 
\item The $p$-symmetric number $p^k-1$ has shift factor $e=0$.
\item The $p$-symmetric number $p^m+1$ has shift factor $e\le m-1$. 
E.g., $(10001)_p$ has shift factor $e=3$.
\item The $5$-symmetric number $(301)_5$ has shift factor $e=1$.
The $2$-symmetric number $(110011)_2$ has shift factor $e=2$.
\item   
The $p$-symmetric number $(100010001)_p$ has shift factor $e=3$.
\end{enumerate}
\end{example}

\begin{proposition}\label{P:k-e}
For any $p$-symmetric number $\nu$ with minimal factorization \\
$\nu *w = (p^k-1)\ell$ with
shift factor $e$, we have $k\ge e+1$ and $\ell \le \ell p^{k-e-1}<\nu$.
\end{proposition}

\begin{proof}
Write $w=p^kw_++w_-$ with $0\le w_-<p^k, w_+\ge 0$, and let $w'=w_++w_-$. 
Then $$\nu w'=\nu(w_+ + w_-)=(p^k-1)(\ell-\nu w_+).$$
Since $s_p(w')=s_p(w_-+w_+)\le s_p(w_+)+s_p(w_-)=s_p(w)$, 
we have $$s_p(\nu w')=k(p-1)\le s_p(\nu)s_p(w')\le s_p(\nu) s_p(w)=s_p(\nu w)=k(p-1).$$ 
Hence equality holds throughout. In particular, $s_p(\nu w')=s_p(\nu)s_p(w')$. 
By Proposition \ref{P:2.1}, the product $\nu w'$ is $p$-adic carry-free.
By the minimality of $w$, we must have $w=w'$. It follows that $w_+=0$. Hence 
$w=w_-<p^k$. 
Now $w=(w_e\cdots w_0)_p$ with $w_e\ne 0$, so $p^k>w\ge p^e$ and hence $e\le k-1$.
That is $k\ge e+1$. 

Finally, we have $\nu=\frac{p^k-1}{w}\ell >p^{k-e-1}\ell\ge \ell$. 
\end{proof}

\subsection{Solution value to $p$-adic change-making problems}

For the rest of this subsection 
we assume our coin set $\cC=\cC(\bi, a)$ (as introduced in \eqref{E:coin_set})
satisfies the following condition:
$\nu\in\bi$ is $p$-symmetric whose minimal factorization is
$\nu*w=(p^k-1)\ell$, $s_p(\nu)=\max_{i\in\bi}s_p(i)$, and $k|a$.

Suppose $\cC=\cC(\bi,a)$ with $a=k$.
Then the $p$-adic form $w=\sum_{j=0}^{k-1}w_jp^j$ of $w$ 
encodes a solution to the change-making problem for $(p^k-1)\ell$.
Let $t_{ij}=w_j$ for $i=\nu, j=0,\ldots,k-1$, and $t_{ij}=0$ otherwise.
By Key Lemma \ref{L:key}, we have
$M_{\cC}((p^k-1)\ell) = \sum_{i\in\bi}\sum_{j=0}^{k-1} t_{ij}=\sum_{j=0}^{k-1} w_j=s_p(w)=\frac{k(p-1)}{s_p(\nu)}$.
Moreover, by considering $\nu * (wp^r)$, we also have
$M_\cC(p^r(p^k-1)\ell) = \frac{k(p-1)}{s_p(\nu)}$
so long as 
$r+e \le a-1$, where $e$ is the shift factor of $\nu$.
Should $r$ exceed $a-1-e$,
then the coins in the representation encoded by $wp^r$
become too large, and fall out of the scope of $\cC$.

Suppose $\cC=\cC(\bi,a)$ with $a=bk$ for some $b\ge 1$.
The minimal factorization $\nu*w=(p^k-1)\ell$ gives
$\nu*w'=(p^a-1)\ell
$
where $w' = w(1+p^k + \dots + p^{(b-1)k})$
(recall Example \ref{ex:p-symmetry} (2)).
The $p$-adic digits of $w'$ encode a solution to the change-making problem of $(p^a-1)\ell$ in the same manner as that above for $w$.
Moreover, so long as $r \le k-e-1$, we will have $p^r w' < p^a$,
and so Key Lemma \ref{L:key} guarantees
that $M_\cC(p^r(p^a-1)\ell) = \frac{s_p(p^r(p^a-1)\ell)}{s_p(\nu)}$.

We summarize these observations in the following proposition:

\begin{proposition}
\label{P:3.10}
Let $\cC=\cC(\bi,a)$ be as above.
Let $e$ be the shift factor of $\nu$.
Then for $r \in \{0,1, \dots, k-e-1\}$,
\[
    M_\cC(p^r(p^a-1)\ell) = \frac{s_p(p^r(p^a-1)\ell)}{s_p(\nu)}
    =\frac{a(p-1)}{s_p(\nu)}.
\]
\qed
\end{proposition}
The following proposition pays special attention to the case where
$\nu = p^k-1$, offering a characterization of
when the lower bound is achieved.

\begin{proposition}\label{P:pk-1-case}
Let $\cC=\cC(\bi,a)$ be as above, where
$\nu=p^k-1 \in \bi$ is the unique maximal $p$-adic weight element.
    Then we have 
    $M_\cC((p^a - 1)\ell) = \frac{a(p-1)}{s_p(\nu)}$
    if and only if $\ell\in\{1,p,\ldots,p^{k-1}\}$.
\end{proposition}

\begin{proof}
Since $k|a$ we write $a=bk$ for some $b\in\Z$.
Notice that $\nu$ is $p$-symmetric with shift factor $0$.
The sufficient direction thus follows from Proposition \ref{P:3.10}.
   
    Conversely, suppose $M_\cC((p^a-1)\ell) = \frac{a(p-1)}{s_p(\nu)}$.
    Then 
    $M_\cC((p^a-1)\ell) = \frac{s_p((p^a-1)\ell)}{s_p(\nu)}$ and $\ell<p^a$.
    Applying Key Lemma \ref{L:key}, there exists a unique solution 
    $\bt=(t_{ij})_{(i,j)\in I}$ to the change-making problem of $(p^a-1)\ell$,
    and we have a carry-free factorization
    \begin{equation}\label{E:factor}
       (p^a - 1)\ell = (p^k-1)w
    \end{equation}
    where $w=\sum_{j=0}^{a-1}t_{p^k-1,j}p^j$ is in its $p$-adic form and $w<p^a$.
    Then
    \[
        w=\sum_{j=0}^{a-1} t_{p^k - 1, j} p^j=\frac{(p^a - 1)\ell}{p^k-1}
    \]
    and we have
    \begin{equation}\label{E:bt1}
      s_p\left(\frac{(p^a-1)\ell}{p^k-1}\right)
      = s_p(w)=\sum_{j=0}^{a-1}t_{p^k-1,j}=M_{\cC}((p^a-1)\ell)
       = \frac{a(p-1)}{k(p-1)}=b.
    \end{equation}
    Since $w<p^a$, \eqref{E:factor} implies $\ell<p^k$,
    and so the product $\ell (1+p^k+\ldots+p^{(b-1)k})$ must be carry-free. 
    Applying Proposition \ref{P:2.1} we now have 
    $s_p(\ell)s_p(1+p^k+\ldots+p^{(b-1)k})=s_p(\ell(1+p^k+\ldots+p^{(b-1)k}))$. 
    Combined with \eqref{E:bt1}, we see that
    $$s_p(\ell)=\frac{s_p(\frac{\ell(p^a-1)}{p^k-1})}{s_p(1+p^k+\ldots+p^{(b-1)k})}=\frac{b}{b}=1.
   $$ 
   Thus $\ell$ is a $p$-power.
   Since $\ell<p^k$, we finish by observing that
   $\ell\in\{1,p,\ldots,p^{k-1}\}$.
 \end{proof}

\subsection{Minimizers and heights}
\label{S:4.3}
This final technical preparation aims to handle the
permutations that arise from a characteristic power series
that lies at the heart of our Dwork-theoretic approach.

For the remainder of this section, we fix a coin set
$\cC=\cC(\bi,a)$ for some $a\ge 1$ and $\bi$ such that 
$\bi$ contains a unique element $\nu$
with maximal $p$-adic weight.

\begin{definition}
\label{D:minimizer}
Let $\bSig$ denote the set of all pairs $(\cI,\sigma)$ 
where $\cI$ is a finite nonempty set of positive integers
and $\sigma$ is a permutation of $\cI$.
A pair $(\cI,\sigma)$ in $\bSig$ is a {\bf minimizer}  if 
$M_\cC(p^a\ell - \sigma \ell) = \frac{s_p(p^a \ell -
    \sigma\ell)}{s_p(\nu)}$ for every $\ell \in \cI$. 
Let $$\Sb \coloneqq \{(\cI,\sigma)\in \bSig \mid (\cI,\sigma) \mbox{ is a minimizer}\}.$$
For each positive integer $m$, 
let 
$$\bSig_m \coloneqq \{(\cI,\sigma)\in\bSig \mid \#\cI=m\}.
$$
\end{definition}

The following proposition is a straightforward consequence of
Definition \ref{D:minimizer}.
\begin{proposition}
\label{P:hereditary}
Let $(\cI,\sigma)\in \bSig$.
Suppose $\sigma=\prod_i \sigma_i$ where $\sigma_i$'s are (disjoint) permutations on sets $\cI_i$, where $\cI$ is the disjoint union of $\cI_i$'s.
Then $(\cI,\sigma)$ is a minimizer if and only if all $(\cI_i, \sigma_i)$ are minimizers.
In particular, if $\sigma_i$'s are the decomposition cycle factors of $\sigma$, and $\cI_i$ is the underlying integers in the cycle $\sigma_i$,
then $(\cI,\sigma)$ is a minimizer if and only if
all $(\cI_i,\sigma_i)$ are cyclic minimizers. \qed
\end{proposition}

\begin{proposition}
\label{P:bounded_min}
\-
\begin{enumerate}

\item  If $(\cI,\sigma)\in \Sb$ then $\cI\subseteq \{1,2,\ldots,\nu\}$.

\item Define a relation on $\Sb$ as follows: $(\cI,\sigma)\le (\cI',\sigma')$ if 
$\cI \subseteq \cI'$ and $\sigma$ is a decomposition factor of $\sigma'$. 
Then $\Sb$ is partially ordered under $\le$, and if it is nonempty,
then there is a unique maximum $(\Ib,\sbb)$.
\end{enumerate}
\end{proposition}

\begin{proof}
(1) Suppose there exists $\ell\in \cI$ such that $\ell>\nu$.
Then we observe that there exists $\ell'\in \cI$ such that 
$\ell'>\nu$ and $\sigma(\ell')\le \ell'$. 
Thus $p^a\ell'-\sigma\ell' >\nu(p^a-1)$.
But by Key Lemma \ref{L:key},
there is a carry-free factorization
$p^a\ell' - \sigma\ell' = \nu * w$, with $w < p^a$.
This implies $p^a\ell' - \sigma\ell' \le \nu(p^a-1)$,
which is a clear contradiction.

(2) It is easy to see that `$\le$' defines a partial ordering on $\Sb$;
we omit the routine verification and continue to showing the existence
of a unique maximal element.
By Proposition \ref{P:hereditary}, it suffices to show that
any two distinct cyclic minimizers must have disjoint underlying sets $\cI$.

If $\#\Sb \in \{0, 1\}$ then there is nothing to prove,
so suppose $\#\Sb \ge 2$, and let $(\cI_1, \sigma_1), (\cI_2,\sigma_2) \in \Sb$
be distinct cyclic minimizers.
We wish to prove that $\cI_1 \cap \cI_2 = \varnothing$, so
suppose for contradiction that $\ell\in\cI_1 \cap \cI_2$.
Let $m\ge 1$ be the smallest 
    such that $\sigma_1^m(\ell)\ne \sigma_2^m(\ell)$. 
(If no such $m$ exists, then $\sigma_1=\sigma_2$ and then $\cI_1=\cI_2$, which contradicts our hypothesis.)   
    Write $\ell'=\sigma_1^{m-1}(\ell)$ --- with $\ell' = \ell$ if $m=1$.
    Our hypothesis implies by the Key Lemma 
    that $\nu|(p^a\ell'-\sigma_i\ell')$ for $i=1,2$ hence $\nu|(\sigma_1\ell'-\sigma_2\ell')$.
   On the other hand, we have
   $|\sigma_1\ell'-\sigma_2\ell'|\le \nu-1$ since 
   $\sigma_i\ell'\in\{1,\ldots,\nu\}$ by part (1) above.
   Thus it must follow that $\sigma_1\ell' = \sigma_2\ell'$,
   contrary to assumption.
Thus $\cI_1\cap\cI_2=\varnothing$,
and the proof is complete.
\end{proof}

\begin{definition}
\label{D:height}
Suppose $\Sb$ is non-empty.
Then by Proposition \ref{P:bounded_min} there is a unique maximal
minimizer $(\Ib,\sbb)$, 
and we call the positive integer 
$t(\cC):=\#\Ib$ the {\bf minimizer height} of $\cC$.
\end{definition}

\begin{proposition}
\label{P:uniq-perm}
Let $\cC=\cC(\bi,a)$ where $\bi$ contains a unique element $\nu$ with maximal $p$-adic weight. 
\begin{enumerate}
\item If $\nu$ is $p$-symmetric with minimal factorization $\nu * w=(p^k-1)\ell$
and $k|a$, then $\Sb\neq \varnothing$. 
\item Suppose $\nu<p^a$. If $\Sb\neq \varnothing$, 
then $\nu$ is $p$-symmetric. 
\end{enumerate}
\end{proposition}

\begin{proof}
(1) 
By Proposition \ref{P:3.10}, the pair $(\{\ell\},\id)$
is a minimizer (where $\id$ is the identity permutation). 
That is $(\{\ell\},\id)\in \Sb$.  

(2) 
For any given minimizer in $\Sb$,
by Proposition \ref{P:hereditary}, all of its cycle decomposition factors
are also in $\Sb$.
So we may let  $(\cI,\sigma)$ be a cyclic minimizer in $\Sb$,
with $\#\cI=m$ and $\sigma$ being an $m$-cycle.   
We then have
$$\left\{
\begin{array}{rcl}
p^a\ell-\sigma\ell&=& \nu *w_{m-1}\\
p^a\sigma\ell-\sigma^2\ell &=& \nu *w_{m-2}\\
&\vdots &\\
p^a\sigma^{m-1}\ell-\ell &=& \nu *w_0
\end{array}\right.
$$
for some $0\le w_0,\ldots,w_{m-1} < p^a$ such that 
$\nu * w_{m-1}, \ldots, \nu *w_1, \nu *w_0$ are all carry-free products by Key Lemma \ref{L:key}.
Write $w\coloneqq p^{a(m-1)}w_{m-1}+\cdots +p^a w_1+w_0$. 
Since $w_i<p^a$ for all $i$, 
the product $\nu w$ is also carry-free.
Now,
\begin{align*}
\nu * w &=\nu (p^{a(m-1)}w_{m-1}+\cdots +p^a w_1 + w_0) \\
     &=p^{a(m-1)}(\nu *w_{m-1})+\cdots +p^a(\nu *w_1)+(\nu *w_0)\\
     &=p^{a(m-1)}(p^a\ell-\sigma\ell)+ \cdots +p^a(p^a\sigma^{m-2}\ell-\sigma^{m-1}\ell)+(p^a\sigma^{m-1}\ell-\ell)\\
     &=p^{am}\ell -\ell.
\end{align*}
That is, we have a carry-free factorization $\nu * w = (p^{am}-1)\ell$, where 
$\ell <p^{am}$ under the hypothesis,
so $\nu$ is a $p$-symmetric number.
\end{proof}

\begin{proposition}
\label{P:bounded_min2}
Let $\cC$ be as in Proposition \ref{P:uniq-perm}.
\begin{enumerate}
\item
If $\nu$ is $p$-symmetric with minimal factorization 
$\nu * w = (p^k-1)\ell$ where $k|a$ and has shift factor $e$,  
then $k-e\le t(\cC)\le \nu$. 
\item 
If $\nu=p^k-1$, $k|a$, then $t(\cC)=k$.
\end{enumerate}
\end{proposition}

\begin{proof}
(1) 
One only needs to observe from Proposition \ref{P:3.10}
that under the hypothesis, 
$\ell,\ell p,\ldots,\ell p^{k-e-1}$ all lie in $\Ib$. Thus $t(\cC)\ge k-e$. 
Proposition \ref{P:bounded_min}(1) shows that  $t(\cC)=\#\Ib\le \nu$. 
\\
(2)
By Proposition \ref{P:pk-1-case}, we see that $(\cI,\id)\in\Sb$ 
implies that $\cI\subseteq \{1,p,\ldots,p^{k-1}\}$.
Suppose $\sigma\ne \id$.
We shall prove that $(\cI,\sigma)\not\in\Sb$.
Choose $\ell\in \cI$ so that 
$\sigma\ell<\ell$,
and suppose 
$M_\cC(p^a\ell-\sigma\ell)=s_p(p^a\ell-\sigma\ell)/s_p(\nu)$.
Then $p^a\ell-\sigma\ell=\nu w$ for some $w$ by Key Lemma
\ref{L:key}. 
This implies $\nu| [(p^a-1)\ell+(\ell-\sigma\ell)]$, and so
$\nu|(\ell-\sigma\ell)$. Since $\ell-\sigma\ell\in\Z_{>0}$,
we then have $\ell>\nu$. 
Therefore $p^a\ell-\sigma\ell>(p^a-1)\ell>(p^a-1)\nu$. By Key Lemma \ref{L:key}, this 
implies $M_\cC(p^a\ell-\sigma\ell)>s_p(p^a\ell-\sigma\ell)/s_p(\nu)$,
contradicting the above hypothesis.
This proves that $(\cI,\sigma)$ is not in $\Sb$.
Thus $\Ib=\{1,p,\ldots,p^{k-1}\}$ and $t(\cC)=k$.
\end{proof}

\begin{remark*}
By Proposition \ref{P:k-e}, we always have $k-e\ge 1$,
so the lower bound in the above Proposition \ref{P:bounded_min2}(1)
is meaningful. 
The minimizer height $t(\cC)$ will be used in Section 
\ref{S:proof}
in determining the multiplicity of the first slope of the Newton polygon. 
\end{remark*}

\section{\texorpdfstring{$p$}{p}-adic estimates of characteristic series}
\label{S:estimate}

The purpose of this section is to prove
Corollary \ref{C:slope}.
It is the essential ingredient for the proofs of the main results
in Section \ref{S:proof}.

\subsection{\texorpdfstring{$p$}{p}-adic change-making and 
\texorpdfstring{$p$}{p}-adic analysis}
\label{S:dwork}

We shall apply change-making (Section \ref{S:change-making}) to the $p$-adic setting in Dwork theory.
First of all, we recall Dwork theory (\cite{Dw64}) following the spirit of \cite{Bom66}.
Let $q=p^a$ for some $a\ge 1$.
Let $\Q_q$ be the degree $a$ unramified extension over the $p$-adic
rational numbers $\Q_p$, and let $\Z_q$ be its ring of integers.
Let $\Omega_1 = \Q_p(\zeta_p)$
where $\zeta_p$ is any fixed primitive $p$-th root of unity,
and let $\Omega_a=\Q_q(\zeta_p)$. For each $z
\in \F_q$, we denote $\hat{z}$
to be the Teichm\"uller lifting in $\Z_q$ of $z$.  
Additionally, let $\tau : \Omega_a \to \Omega_a$ be the
lift of the Frobenius endomorphism of $\F_q$ which fixes $\Omega_1$.
In particular, we have $\tau(\hat{z})=\hat{z}^p$.
Let
$E(x) = \exp \left(\sum_{i=0}^{\infty}\frac{x^{p^i}}{p^i}\right)$ be the 
    $p$-adic Artin-Hasse exponential function, which lies in $\Z_p[[x]]$.
    Pick and fix a root $\gamma\in\Omega_1$ of $\log E(x)$
satisfying $v_p(\gamma) = 1/(p-1)$. 
Since $\gamma$ is a uniformizer, and $\Z_q[\zeta_p]=\Z_q[\gamma]$, 
our $p$-adic estimate is conveniently reduced to a $\gamma$-adic estimate.

Fix a polynomial $f(x)=\sum_{1\le i\le d,p\nmid i }a_i x^i\in \F_q[x]$.
Let 
$F(x)\coloneqq \prod_{i=1}^{d} E(\gamma\hat a_i x^i )$. 
Let 
$F_0(x) \coloneqq\prod_{j=0}^{a-1} \tau^{j}(F(x^{p^j}))$.
Notice that $F_0(x)$ lies in the following $\Omega_a$-Banach algebra:
$$
\sL=\left\{\sum_{n=0}^\infty A_n x^n\in \Omega_a[[x]]\;|\;
v_p(A_n)\ge \frac{pn}{dq(p-1)}+O(1)\right\}. 
$$

Set $I\coloneqq\Supp(f)\times \{0,1,\ldots,a-1\}.$ 
We consider the $p$-adic change-making problem, introduced in Section \ref{S:change-making},  
with the following coin set:
\begin{equation}\label{E:coin-set2}
\fbox{$\cC=\cC(\Supp(f),a)=\{ip^j\mid (i,j)\in I\}=\{ip^j\mid i\in\Supp(f), j=0,\ldots,a-1\}.$}
\end{equation}
For the readers' convenience, we recall some notation from 
Section \ref{S:change-making}:
For every integer $N\ge 0$, 
$\cT(N)$ is the representation set defined in \eqref{E:I(N)},
and $M_\cC(N)$ is the solution value to the $p$-adic change-making 
problem.
Because $\cC$ satisfies the same structure provided in
Section \ref{S:change-making},
Key Lemma \ref{L:key} applies.
We will immediately make use of this in the lemma that follows.

\begin{lemma}
\label{L:G_N1}
Write
$E(x)=\sum_{n=0}^{\infty} \beta_n x^n$ and 
$F_0(x)=\sum_{N=0}^{\infty} G_N x^N$.
\begin{enumerate}
\item Then 
$
G_N = \sum_{\bt \in \cT(N)}\kappa(\bt) \gamma^{|\bt|},
$
where 
$
\kappa(\bt)= \prod_{(i,j)\in I}\beta_{t_{ij}}\hat a_i^{p^j t_{ij}}.
$
\item Suppose $\nu$ is a maximal $p$-adic weight element in $\Supp(f)$.
Then 
$$v_p(G_N)
\ge \frac{M_\cC(N)}{p-1}
\ge \frac{s_p(N)}{(p-1)s_p(\nu)}.$$ 
If $\nu$ is the unique maximal $p$-adic weight element and the second inequality is an equality, then 
$v_p(G_N)=\frac{s_p(N)}{(p-1)s_p(\nu)}$.
\end{enumerate}

\end{lemma}

\begin{proof}
(1) By definition, 
\[
    F_0(x) = \prod_{j=0}^{a-1}\prod_{i\in\Supp(f)}
    E(\gamma \hat{a}_i x^{ip^j})^{\tau^j} =
    \prod_{(i,j)\in I} E(\gamma \hat a_i^{p^j} x^{ip^j}).
\]
Since each factor has expansion
$E(\gamma \hat a_i^{p^j} x^{ip^j})=
\sum_{t_{ij}=0}^{\infty} ( \beta_{t_{ij}}\hat a_i^{p^jt_{ij}}\gamma^{t_{ij}})x^{ip^jt_{ij}},
$
the $x^N$-coefficient $G_N$ of the product is
$
    \sum 
    \left( \prod_{(i,j) \in I} \beta_{t_{ij}}\hat a_i^{p^j t_{i,j}}\right)
    \gamma^{|\bt|},
$
where the sum ranges over all $\bt\in\cT(N)$. 

(2) The second inequality follows from Key Lemma \ref{L:key}.
Since $v_p(\kappa(\bt))\ge 0$ and $v_p(\gamma) = \frac{1}{p-1}$,
we have $v_p(G_N)\ge \frac{1}{p-1}\min_{\bt\in \cT(N)}|\bt|=\frac{M_{\cC}(N)}{p-1}$. This proves the first inequality. 
By Key Lemma \ref{L:key}, if $\nu$ is the unique maximal element
and the second equality holds, 
then there is a unique solution $\bt=(t_{ij})\in \cT(N)$ with minimal $|\bt|=M_\cC(N)$.
This yields a unique lowest $\gamma$-power term in Part (1)'s expression of $G_N$,
that is, $\kappa(\bt)\gamma^{M_{\cC}(N)}$. 
Moreover, since $t_{ij}\in\{0,1,\ldots,p-1\}$, we have $\beta_{t_{ij}}=\frac{1}{t_{ij}!}$ hence 
$v_p(\beta_{t_{ij}})=0$.
This implies $v_p(\kappa(\bt))=0$. Therefore,
$v_p(G_N)=\frac{M_\cC(N)}{p-1}$.
Combined with the hypothesis that $M_\cC(N)=\frac{s_p(N)}{s_p(\nu)}$, we have 
$v_p(G_N)=\frac{s_p(N)}{(p-1)s_p(\nu)}$.
\end{proof}

For our result below, we recall  the following notation from Section \ref{S:4.3}:
For every $m\ge 1$, $\bSig_m$ is the set of all pairs $(\cI, \sigma)$
where $\cI$ is a set of $m$ prime-to-$p$ positive integers and $\sigma$ 
is a permutation of $\cI$. A pair $(\cI,\sigma)$ is a minimizer if 
the change-making solution value 
$M_\cC(p^a\ell-\sigma\ell)$ achieves its minimum 
$s_p(p^a\ell-\sigma\ell)/s_p(\nu)$
for every $\ell\in \cI$ at the same time. 

\begin{lemma}
\label{L:G_N}
Let $G_N$ be as in Lemma \ref{L:G_N1}.
Suppose $\nu$ is a maximal $p$-adic weight element in $\Supp(f)$.
For any $m\ge 1$, if $(\cI,\sigma)\in \bSig_m$, then  
    $
        v_p\left(\prod_{\ell\in \cI} G_{q\ell - \sigma \ell}
            \right) \ge \frac{am}{s_p(\nu)}.
    $
If $\nu$ is the unique maximal $p$-adic weight element in $\Supp(f)$ and
$\nu<q$, then 
$(\cI,\sigma)$ is a minimizer if and only if  
$v_p\left(\prod_{\ell\in \cI} G_{q\ell - \sigma \ell}
            \right) = \frac{am}{s_p(\nu)}$.
\end{lemma}

\begin{proof}
If the product contains $G_{q\ell-\sigma \ell}$ with $q\ell-\sigma\ell<0$
then the product equals $0$, and the first statement always holds; 
since such a $(\cI,\sigma)$ is not a minimizer, the second statement holds vacuously. 
From now on we assume all sub-indices in the product we consider 
are positive.

By Lemma \ref{L:G_N1}(2) we have 
   \begin{equation*}
        \sum_{\ell\in \cI} v_p(G_{q\ell- \sigma \ell}) 
                \ge \sum_{\ell\in \cI} \frac{M_\cC(q\ell - \sigma \ell)}{p-1} 
                \ge 
                \frac{\sum_{\ell\in \cI} s_p(q\ell - \sigma \ell)}{(p-1)s_p(\nu)}.
    \end{equation*}
By Proposition \ref{lem:sp_qi-j}, 
\begin{equation*}
\sum_{\ell\in \cI} s_p(q\ell - \sigma \ell)
\ge 
\sum_{\ell\in \cI} (a(p-1)+s_p(\ell -1 ) -s_p(\sigma(\ell)-1))  
=am(p-1).
\end{equation*}
Combining the above two inequalities, we obtain the desired inequality.

It remains to prove the second statement.
First we prove the `if' direction.
By Lemma \ref{L:G_N1}(2) again, 
the equality 
$$
v_p(\prod_{\ell\in \cI}G_{q\ell-\sigma\ell})=\frac{am}{s_p(\nu)}
$$
implies that
\[
    v_p(G_{q\ell - \sigma\ell})=
    \frac{M_\cC(q\ell - \sigma\ell)}{p-1}
    =
    \frac{s_p(q\ell-\sigma\ell)}{(p-1)s_p(\nu)}
\]
for all $\ell\in \cI$.
Therefore
$M_\cC(q\ell - \sigma \ell) = \frac{s_p(q\ell - \sigma \ell)}{s_p(\nu)}$
for all $\ell \in \cI$; i.e., $(\cI,\sigma)$ is a minimizer.

To show the `only if' direction:
Suppose $(\cI,\sigma)$ is a minimizer. This implies by Lemma \ref{L:G_N1}(2)
that $v_p(G_{q\ell-\sigma\ell})=\frac{s_p(q\ell-\sigma\ell)}{(p-1)s_p(\nu)}$ 
for every $\ell\in \cI$. Since  $\ell,\sigma\ell\le \nu<q$, then by Proposition \ref{lem:sp_qi-j}
we have $s_p(q\ell-\sigma\ell)
=a(p-1)+s_p(\ell-1)-s_p(\sigma\ell-1)$.
Summing these up over $\ell\in \cI$
we have
\[
    v_p\left(\prod_{\ell\in \cI}G_{q\ell-\sigma\ell}\right)
    =\frac{\sum_{\ell\in \cI}(a(p-1)+s_p(\ell-1)-s_p(\sigma\ell-1))}
    {(p-1)s_p(\nu)}=\frac{am}{s_p(\nu)}.
\]
\end{proof}

\subsection{Fredholm determinant of Dwork operator}\label{S:series}

Let the Dwork operator $\alpha: \sL\longrightarrow \sL$ be defined by $\alpha\coloneqq \Phi_q\circ F_0(x)$ where
$    \Phi_q \left(\sum_i A_i x^i\right) = \sum_i A_{qi}x^i$
and $\alpha(\sum_i A_i x^i)= \Phi_q(F_0(x)\cdot \sum_{i}A_i x^i)$
for any $\sum_i A_i x^i\in \sL$. Then $\alpha$ is a 
$\Omega_a$-linear completely continuous operator on $\sL$.

Let $\alpha^\triangle$ denote the restriction of $\alpha$ to the Banach subspace 
$x\sL$, and write $\cB=\{1,x,x^2,\dots\}$ and $\cB^\triangle=\{x,x^2,\dots\}$ as
formal bases for $\sL$ and $x\sL$, respectively.
Let $\Mat(\alpha^\triangle)$ be the matrix of $\alpha^\triangle$
with respect to the basis $\cB^\triangle$;
that is, $\Mat(\alpha^\triangle)=(G_{qi-j})_{i,j\ge 1}$,
where $G_{qi-j}$'s 
are coefficients of $F_0$ in Lemma \ref{L:G_N1}.
We have the following Fredholm determinant:  

\begin{equation}\label{E:C^*}
C^\triangle(f/\F_q,s)\coloneqq\det(1-s\Mat(\alpha^\triangle))=1+C_1s+C_2 s^2+\cdots
\end{equation}
where each coefficient is a sum over all $\bSig_m$
(see Definition \ref{D:minimizer}) as follows:
\begin{equation}
\label{E:C_m} 
C_m=\sum_{(\cI,\sigma)\in \bSig_m}(-1)^m\sgn(\sigma)\prod_{\ell\in \cI}G_{q\ell-\sigma\ell}.   
\end{equation}
Below we apply our result from Section \ref{S:symmetric} (in particular, Propositions \ref{P:uniq-perm}
and \ref{P:bounded_min2}) to explicit $p$-adic estimates.

\begin{proposition}\label{P:bound}
Let $\cC=\cC(\Supp(f),a)$ be the coin set as in \eqref{E:coin-set2}. 
\begin{enumerate}
\item 
Let $\nu$ be a maximal $p$-adic weight element in $\Supp(f)$. Then     
$
        v_p(C_m) \ge \frac{am}{s_p(\nu)}
    $ for all $m\ge 1$.   
\item 
Suppose $\nu$ is the unique maximal $p$-adic weight element in $\Supp(f)$ and $\nu<q$.
\begin{enumerate}
\item If $\nu$ is $p$-symmetric with minimal factorization
    $\nu * w = (p^k-1)\ell$ and $k|a$,
then $\Sb\neq \varnothing$.
Let $t \coloneqq t(\cC)$
be the minimizer height of $\cC$,
and let $e$ denote the shift factor of $\nu$.
Then we have $k-e\le t\le \nu$, $v_p(C_t)=\frac{at}{s_p(\nu)}$, and 
$v_p(C_m)>\frac{am}{s_p(\nu)}$ for all $m>t$.
\item 
Conversely, if $v_p(C_m)=\frac{am}{s_p(\nu)}$ for some $m\ge 1$, then 
$\nu$ is $p$-symmetric.    
\end{enumerate}
\item Suppose $p^k-1$ with $k|a$ is the unique maximal $p$-adic weight element in $\Supp(f)$.
Then $v_p(C_m) \ge \frac{am}{k(p-1)}$ for all $m\ge 1$,
$v_p(C_k)=\frac{a}{p-1}$, 
and $v_p(C_m)>\frac{am}{k(p-1)}$ for all $m>k$.
\end{enumerate}
\end{proposition}

\begin{proof}
(1) By Lemma \ref{L:G_N},
$ v_p(\prod_{\ell\in \cI}
    G_{q\ell - \sigma \ell})\ge \frac{am}{s_p(\nu)}
$ for every $(\cI,\sigma)\in \bSig_m$.
So their sum still has 
$v_p(C_m) \ge \min_{(\cI,\sigma)\in\bSig_m} v_p(\prod_{\ell\in \cI}
    G_{q\ell - \sigma \ell})\ge \frac{am}{s_p(\nu)}.$
\\
(2a)
$\Sb \ne \varnothing$ immediately follows from
Proposition \ref{P:uniq-perm}(1),
and the inequality $k-e \le t \le \nu$ immediately follows from
Proposition \ref{P:bounded_min2}(1).

By \eqref{E:C_m}, we have 
$
C_t=C_{t,-} + C_{t,+}
$
where 
\begin{equation*}
    C_{t,+}= 
    \sum_{\stackrel{(\cI,\sigma)\in \bSig_t}{(\cI,\sigma)\ne (\Ib,\sbb)}}
    (\pm)\prod_{\ell\in \cI} G_{q \ell - \sigma\ell},\qquad
C_{t,-}=\pm \prod_{\ell\in \Ib}G_{q\ell-\sbb\ell}.
\end{equation*}
To prove that $v_p(C_t)=\frac{at}{s_p(\nu)}$, it suffices to show
the inequality
$v_p(C_{t,+})>\frac{at}{s_p(\nu)}$
and the equality
$v_p(C_{t,-})=\frac{at}{s_p(\nu)}$.   

Let  $(\cI,\sigma)\in \Sigma_t$ and $(\cI,\sigma)\neq (\Ib,\sbb)$.
Since $(\Ib,\sbb)$ is the unique maximal minimizer,
$(\cI,\sigma)$ is not a minimizer.
By Lemma \ref{L:G_N}, we have
 $
        v_p\left(\prod_{\ell \in \cI} G_{q\ell - \sigma\ell}\right)
            > \frac{at}{s_p(\nu)}
$,
hence
 $v_p(C_{t,+})>\frac{at}{s_p(\nu)}$.
On the other hand, by Lemma \ref{L:G_N}, since $(\Ib,\sbb)$ is a minimizer,
we have
$v_p(C_{t,-})=
v_p(\prod_{\ell\in \Ib}G_{q\ell-\sbb\ell}) = \frac{at}{s_p(\nu)}$.

Suppose $m>t$.
Then $\bSig_m$ contains no minimizers by the maximality of $(\Ib,\sbb)$, 
hence \eqref{E:C_m} and Lemma \ref{L:G_N} show that 
$v_p(C_m)>\frac{am}{s_p(\nu)}$.

\noindent 
(2b)
Suppose $v_p(C_m)=\frac{a m}{s_p(\nu)}$ for some $m\ge 1$.
From (\ref{E:C_m}) and Lemma \ref{L:G_N},
there must exist a pair $(\cI,\sigma)\in \bSig_m$ 
that is a minimizer.
By  Proposition \ref{P:uniq-perm}(2), since $\nu<q$, 
we conclude that $\nu$ is $p$-symmetric.    
\\
(3) 
Notice that  $\nu=p^k-1$ is $p$-symmetric.
The minimizer height $t(\cC)=k$ by Proposition \ref{P:bounded_min2}.
The statement follows from the above argument. 
\end{proof}

\begin{corollary}
\label{C:slope}
Let $\NP_q(C^\triangle(f/\F_q,s))$
denote the $q$-adic Newton polygon of the power series 
$C^\triangle(f/\F_q,s)$ in \eqref{E:C^*}. 
\begin{enumerate}
\item Then the first slope of  $\NP_q(C^\triangle(f/\F_q,s))$ 
is $\ge \frac{1}{\max_{i\in\Supp(f)}s_p(i)}$.
\item 
Suppose $\Supp(f)$ contains $\nu$ which is the unique element achieving the maximal $p$-adic weight and $\nu<q$.
\begin{enumerate}
\item 
If $\nu$ is a $p$-symmetric number with minimal factorization $\nu* w=(p^k-1)\ell$
and $k|a$, then the first slope of $\NP_q(C^\triangle(f/\F_q,s))$ 
equals $\frac{1}{s_p(\nu)}$. 
In this case, the multiplicity $t_1$ of 
this first slope is equal to the minimizer height $t(\cC(\Supp(f),a))$.
In particular, $k-e\le t_1\le \nu$, where
 $e$ is the shift factor of $\nu$.
\item 
Conversely, if $v_p(C_m)=\frac{am}{s_p(\nu)}$ for some $m\ge 1$, then 
$\nu$ is $p$-symmetric.
\end{enumerate}
\item Suppose $\nu=p^k-1$ with $k|a$ is the unique maximal $p$-adic weight 
element in $\Supp(f)$. Then the first slope of $\NP_q(C^\triangle(f/\F_q,s))$ is $\frac{1}{k(p-1)}$ of multiplicity $k$.
\end{enumerate}
\end{corollary}
\begin{proof}
The first slope of $\NP_q(C^\triangle(f/\F_q, s))$ is given by
$\delta \coloneqq \inf_{m \ge 1} \frac{v_q(C_m)}{m}$,
with the multiplicity given by the largest $t$ such that $\frac{v_q(C_t)}{t} =
\delta$.
Parts (1), (2), and (3) of this corollary now follow directly from
parts (1), (2), and (3), respectively, of Proposition
\ref{P:bound}.
\end{proof}

\section{Proof of the main theorem and its applications}
\label{S:proof}

\subsection{\texorpdfstring{$L$}{L}-functions of exponential sums and their Newton slope}

This section inherits all notation from Section \ref{S:estimate}.
We will additionally set $\zeta_p \coloneqq E(\gamma)$,
noting that it is a primitive $p$-th root of unity in $\bar\Q_p$.

Define the $m$-th exponential sum of $f(x)\in\F_q[x]$ as 
\[S_m(f)=\sum_{c\in\F_{q^m}}\zeta_p^{\Tr_{\F_{q^m}/\F_p}(f(c))}.\]
Then the $L$-function of the exponential sum of $f/\F_q$ is
\[
L(f,s)\coloneqq\exp\left(\sum_{m=1}^\infty S_m(f)\frac{s^m}{m}\right).
\]
Following Dwork theory (see \cite{Bom66}),
there is a relationship between $L(f, s)$ and the Fredholm determinant $\det(1-s\alpha)$.
In particular,
\[
    L(f, s) = \frac{\det(1-s \alpha)}{(1-s)\det(1-qs\alpha)}.
\]
Let $\NP_q$ denote the $q$-adic Newton polygon of a power series
and let $\NP_q^{<1}$ denote the section of the $q$-adic Newton polygon 
$\NP_q$ that has slope $<1$.
By Weil's theorem,
$L(f,s)$ is a polynomial in $\Z_p[\zeta_p][s]$ of degree $d-1$
with Newton slopes $<1$.
An elementary computation shows that 
\begin{equation*}
\NP_q(L(f,s))=\NP_q^{<1}(\det(1-s\Mat(\alpha^\triangle)))=\NP_q^{<1}(C^\triangle(f/\F_q,s)).
\end{equation*}
By \eqref{E:C^*}, we obtain the following: 
\begin{equation}
\label{E:L-function}
\NP_q(L(f,s))=\NP_q^{<1}(C^\triangle(f/\F_q,s))=\NP_q^{<1}(1+C_1s+C_2s^2+\cdots).
\end{equation}
Equivalently, $\NP_q(L(f,s))$ is equal to $\NP_q(1+C_1s+C_2s^2+\cdots)$
up to horizontal length $d-1$.

Recall the notion of $p$-symmetric numbers, minimal factorization, and shift factors
from Definitions \ref{def:p-symmetry} and \ref{D:symmetric_number}:
if $\nu$ is $p$-symmetric, it has unique minimal factorization
$\nu*w=(p^k-1)\ell$, and
the shift factor $e$ is equal to the number of $p$-adic digits of $w$ minus $1$.

\begin{theorem}    
\label{T:book_worm}
Suppose  $f=\sum_{i\ge 1,p\nmid i} a_i x^i \in \F_q[x]$.
Then $\NP_q(L(f, s))$
has its first slope $\ge \frac{1}{\max_{i\in \Supp(f)} s_p(i)}$. 
\begin{enumerate}
\item 
Suppose $\Supp(f)$ contains a
unique element $\nu$ achieving the maximal $p$-adic weight. Then
the first slope achieves the minimum $\frac{1}{s_p(\nu)}$
if and only if $\nu$ is a $p$-symmetric number.
In this case, the multiplicity $t_1$ of this first slope 
satisfies $k-e\le t_1\le \nu$.

\item If $\nu=p^k-1$ is the unique maximal $p$-adic weight element in $\Supp(f)$,
then the first slope of $\NP_q(L(f,s))$ is 
$\frac{1}{k(p-1)}$ of multiplicity $k$.

\end{enumerate}
\end{theorem}

\begin{proof} 
Since the $q$-adic Newton polygon $\NP_q(L(f,s))$ 
of $L(f,s)$ is independent of the choice of base field of $f$, 
we may extend the base field $\F_q$ of $f$ such that 
$q=p^a>\nu$ and such that $k|a$ whenever 
$\nu\in\Supp(f)$ is $p$-symmetric with minimal factorization 
$\nu*w=(p^k-1)\ell$.
By \eqref{E:L-function}, it is reduced to compute the first slope of 
$\NP_q(C^\triangle(f/\F_q,s))$, 
which is done in Corollary \ref{C:slope}.
\end{proof}

\subsection{Zeta functions of Artin-Schreier curves}

The following result will have Theorem
\ref{T:main-thm-tight} as an immediate corollary, and will strengthen
the result by offering bounds on the multiplicity of the first slope.

\begin{theorem}\label{T:main!}
Let $X_f: y^p-y=f(x)$ with $f\in\F_q[x]$ of $\deg(f)=d\ge 3$.
Then the first slope of $X_f$ is
$\ge \frac{1}{\max_{i\in\Supp(f)} s_p(i)}$.
\begin{enumerate} 
\item If $\nu$ is the unique element in $\Supp(f)$ with the maximal $p$-adic weight, 
then the equality holds if and only if $\nu$ is $p$-symmetric.
In this case, if $\nu$ has its minimal factorization $\nu*w=(p^k-1)\ell$
and shift factor $e$, then 
the multiplicity $t_1$ of the first slope satisfies 
$(k-e)(p-1)\le t_1 \le \nu(p-1)$.
\item 
If $\nu=p^k-1$ is the unique maximal $p$-adic weight element in $\Supp(f)$, then 
the first slope is $\frac{1}{k(p-1)}$ of multiplicity $k(p-1)$. 
\end{enumerate}
\end{theorem}

\begin{proof}
Let $g$ denote the genus of $X$. We know $g=\frac{(p-1)(d-1)}{2}$.
It is well-known that 
$$
Z(X/\F_q,s)=\frac{\prod_{\sigma} \sigma(L(f,s))}{(1-s)(1-qs)}
$$
where the product ranges over all
$\sigma\in \Gal(\Q_p(\zeta_p)/\Q_p)$.
The reciprocal zeros $\alpha_1,\ldots,\alpha_{2g}$ of the zeta function are precisely the reciprocal
roots of $\sigma(L(f,s))$ for all $\sigma$.
Each $\sigma(L(f,s))$ has the same $q$-adic Newton polygon as $L(f,s)$ 
for all $\sigma$. Thus its first slope is equal to the first slope of $\NP_q(L(f,s))$, 
and our statements about the first slope follow from Theorem \ref{T:book_worm}. 
Since the $q$-adic Newton polygon of 
the numerator of the zeta function is a dilation of 
that of $L(f,s)$ by a factor of $p-1$,
our statements about the multiplicities follow 
from Theorem \ref{T:book_worm}.
\end{proof}

\begin{remark}
If $\Supp(f)$ does not have a unique element $\nu$ that achieves
the maximal $s_p(-)$, 
the equality in Theorem \ref{T:main!} may or may not hold. 
It depends on 
the actual $f$, not just the support $\Supp(f)$. 
For example, when $p=2$,
the set $\Supp(f)=\{21,19,13,7,3\}$ does not have a unique maximum.
We have $\max_{i\in \Supp(f)}(s_2(i))=3$.
However, the first slope of $X_f$ is $1/2$
when $f=x^{21} + x^{19} + x^{13} + x^7 + x^3$. 
\end{remark}

\begin{remark}[Maximal \(p\)-adic weight set in \(\Supp(f)\) is invariant]
Since Newton polygons of curves are invariant under isomorphisms, it bears
considering whether a curve \(X_f\) that fails to meet the uniqueness criterion
for \(\Supp(f)\) could be isomorphic to a curve \(X_g\) that does meet it.

Suppose $f,g\in \mathbb{F}_q[x]$ are reduced, i.e., every element of
$\Supp(f)\cup \Supp(g)$ is coprime to $p$. Suppose
$X_f$ and $X_g$ are isomorphic as Artin-Schreier covers of $\mathbb{P}^{1}$
preserving $\infty$.
Then the isomorphism is given by
$
x\longmapsto \alpha x+\beta, y\longmapsto cy+h,
$
where 
$\alpha\in \mathbb{F}_q^*$, $\beta\in \mathbb{F}_q$,
$h\in \mathbb{F}_q[x]$, and $c\in \mathbb{F}_p^*$. 
Hence
$
cg(x)=f(\alpha x+\beta)+h(x)-h(x)^p.
$
Let
$
m:=\max_{i\in \Supp(f)}s_p(i)$ and 
$M(f):=\{i\in \Supp(f):s_p(i)=m\}.
$
Fix
$N\geq 1$ with $s_p(N)\geq m$. 
A term $a_w(\alpha x+\beta)^w$ in $f(\alpha x + \beta)$
can contribute to the coefficient of $x^N$ only if $w\geq N$, in which case
the binomial factor is $\binom{w}{N}$. By Kummer's formula,
$v_p\binom{w}{N}
=
\frac{s_p(N)+s_p(w-N)-s_p(w)}{p-1}.
$
Thus $\binom{w}{N}\not\equiv 0\pmod p$ only if
$
s_p(w)=s_p(N)+s_p(w-N).
$
Since $s_p(w)\leq m\leq s_p(N)$, this forces
$
s_p(w)=s_p(N)=m, s_p(w-N)=0.
$
Therefore $w=N$. Consequently,
$[x^N]f(\alpha x + \beta)=
a_N\alpha^N$ if $N\in M(f)$, and equals $0$ otherwise.
On the other hand, if $[x^N](h-h^p)\ne 0$,
then $p\nmid N$ since $g(x)$ is reduced.
Thus, $[x^N]h \ne 0$,
which implies $[x^{p^r N}](h-h^p)\ne 0$
for some $r \ge 1$.
This again is impossible because $g(x)$ is reduced.
We now have that $[x^N](cg) = [x^N]f(\alpha x + \beta) = \alpha^N[x^N]f$.
Therefore the maximal
$p$-adic weight subset of $\Supp(g)$ is $M(f)$.
\end{remark}

It is well-known that the  
first slope of an affine curve $X$ gives the divisibility of 
the number of rational points on $X$ (see \cite[Proposition 2.2]{Ka71}). 
Write $X_f^{\aff}$ for the affine Artin-Schreier curve defined by $y^p-y=f$.
Then our result above (for example Theorem \ref{T:main-thm-tight}) recovers 
the following known divisibility (see \cite{Wan95} for example).
 
\begin{corollary}
If $\nu\in\Supp(f)$ has the maximum $p$-adic weight and $q=p^a$ then 
$$
p^{\pceil{\frac{am}{s_p(\nu)}}} \;| \;
\# X_f^{\aff}(\F_{q^m}) \mbox{ for every $m\ge 1$}.
$$
\end{corollary}

\subsection{Construction of curves with first slope \texorpdfstring{$1/n$}{1/n} 
for every \texorpdfstring{$n\ge 2$}{n≥2} in every characteristic \texorpdfstring{$p$}{p},  and other applications}

We conclude this paper by presenting some corollaries of
Theorem \ref{T:main!}.
Since for every positive integer $n$ there is a $p$-symmetric number $\nu$ with $s_p(\nu)=n$
(see, e.g., Example \ref{ex:p-symmetry}(9) or Remark \ref{R:last}),
the following corollary ---
which follows directly from Theorem \ref{T:main!} ---
enables us to construct many families
of curves with a fixed first slope:
\begin{corollary}
    \label{C:broad-families}
    Fix an integer $n \ge 2$.
    Let $\nu$ be a $p$-symmetric number with $s_p(\nu) = n$,
    and let $g(x) \in \bar\F_p[x]$ be any polynomial satisfying
    $\max_{i \in \Supp(g)} s_p(i) < n$
    and $p \nmid i$ for all $i \in \Supp(g)$.
    Then, setting $f(x) = a_\nu x^\nu + g(x)$
    for some $a_\nu \in \bar\F_p^*$,
    the first slope of $X_f$
    is equal to $1/n$. \qed
\end{corollary}

Recall that it is already known that if $f$ is of degree $d$ with $d< p$,
then the first slope of $X_f$ achieves the lower bound of $1/d$ if and only if
$p\equiv 1\bmod d$.
We provided a generalization of this in the introduction,
and are now ready to prove it here.

\begin{corollary}[Corollary \ref{C:spd-family-1}]
    \label{C:spd-family}
   Suppose $\Supp(f)$ has a unique maximal $p$-adic weight element $\nu$.
If $p\equiv 1\bmod s_p(\nu)$, 
then the first slope of $X_f$ is equal to $1/s_p(\nu)$.
\end{corollary}

\begin{proof}
    By Theorem \ref{T:main!}(1), it remains to show that
    if $s_p(\nu) \mid  (p-1)$, then $\nu$ is $p$-symmetric.
    If $s_p(\nu) = n$, $n | (p-1)$,
    then setting $m = (p-1)/n$,
    it follows that the ($p$-adic) product $\nu m$ is carry-free,
    and $s_p(\nu m) = p-1$.
    If we write $\nu m$ $p$-adically as $\nu m = d_0 + d_1p + \dots + d_r p^r$,
    then we have $d_0 + d_1 + \dots + d_r = s_p(\nu m) \equiv 0 \pmod{p-1}$,
    which implies that $(p-1) \mid \nu m$.
    Thus $\nu m = (p-1)\ell$ for some $\ell > 0$.
    Since $s_p(\nu m)=p-1$, 
    the product $(\nu m) (1 + p + p^2 + \dots + p^{k-1})$ for all $k \ge 1$
   is carry-free.
    Therefore, setting $w = m(1 + p + \dots + p^{k-1})$
    for some $k$ such that $p^k > \ell$,
    we have
    \[
        \nu * w = (1 + p + p^2 + \dots + p^{k-1})(p-1)\ell 
        = (p^k - 1)\ell
    \]
    and so $\nu$ is $p$-symmetric.
\end{proof}

Our next corollary provides explicit constructions of curves $X_f$ 
via the support set of $f$. 
Its first part shows that for any $n \ge 2$,
there exist curves of arbitrarily large genus
with first slope equal to $1/n$.
The second part provides a general upper bound on the smallest genus
for which there exists a curve with first slope $1/n$.
\begin{corollary}
\label{T:main!!}
Fix a prime $p$ and an integer $n\ge 2$.
Let $N>0$.
\begin{enumerate}
\item Let $\nu$ be a $p$-symmetric number with $s_p(\nu)=n$. 
Suppose $\{\nu,1+p^{\pceil{\frac{2N}{p-1}}}\}\subseteq \Supp(f)$, and 
$$\Supp(f)\subseteq \left\{
\nu, 1+p^{\pceil{\frac{2N}{p-1}}}, 1+\sum_{k=1}^{m}p^{i_k}\mid 
1\le m\le n-2, 1\le i_1<i_2<\ldots<i_m 
\right\}.
$$
Then 
$X_f: y^p-y=f$ is of genus $\ge N$
and first slope $\frac{1}{n}$.
\item There is a $p$-symmetric number $\nu$ with $s_p(\nu)=n$ and 
$\nu\le \frac{p^n-1}{p-1}$.
Suppose $\nu\in \Supp(f)$ and $$\Supp(f)\subseteq \left\{
\nu, 1+\sum_{k=1}^{m}p^{i_k}\mid 1\le m\le n-2, 1\le i_1<i_2<\cdots <i_m\le n-1\right\}.$$
Then $X_f:y^p-y=f$ is of genus 
$\le \frac{p^n-p}{2}$ with first slope equal to $\frac{1}{n}$.
\end{enumerate}
\end{corollary}

\begin{proof}

1) The hypothesis says that the set $\Supp(f)$ has a maximal $p$-adic weight element $\nu$, which is unique when $n\ge 3$.
The slope part follows from Theorem \ref{T:main!};
the genus part follows from direct computation: let $d=\deg(f)$,
the genus of $X_f$ is 
$\frac{(d-1)(p-1)}{2}>N$. 

2) 
It is clear that $\frac{p^n-1}{p-1}$ is $p$-symmetric and $s_p(\frac{p^n-1}{p-1})=n$.
Following a similar argument as the above, the first slope of $X_f$ is equal to $\frac{1}{n}$ 
by Theorem \ref{T:main!}. 
Its genus is $g=\frac{(\nu-1)(p-1)}{2}\le \frac{p^n-p}{2}$.
\end{proof}

Our last corollary,
which was known for the case $p=2$ in \cite{SZ02},
gives a specific family of Artin-Schreier curves
where both the first slope and its multiplicity can be completely determined.
\begin{corollary}\label{C:ppp}
Let $f=\sum_{1\le i\le d,p\nmid i} a_ix^i\in\F_q[x]$ be of degree $d \ge 3$.
Suppose that 
    $p^k-1\le d\le 2p^k-p^{k-1}-2$ and $a_{p^k-1}\ne 0$ for some $k \ge 1$.
    Then the first slope of $X_f$ is equal to $\frac{1}{(p-1)k}$
    of multiplicity $(p-1)k$, and $X_f$ is  non-supersingular
    if and only if  $(p,k) \ne (2,2)$.
\end{corollary}

\begin{proof}
Observe that $s_p(p^k-1) = k(p-1)$.
The smallest $\nu' > p^k-1$ such that $s_p(\nu') \ge k(p-1)$
is $\nu' = 2p^k - p^{k-1}-1$, so the upper bound on $d$
guarantees that
$\nu \coloneqq p^k-1\in \Supp(f)$ is the unique element 
with the maximal $s_p(\nu)=k(p-1)$ in $\Supp(f)$.
The first two statements now follow from Theorem \ref{T:main!}.
For the final statement,
notice that $X_f$ is supersingular if and only if its genus is $\ge 1$ and 
the first slope is $\frac{1}{(p-1)k}=\frac{1}{2}$.
This is equivalent to $(p,k)=(2,2),$ 
or $(3,1)$. The latter contradicts our hypothesis.  
\end{proof}

\begin{remark}\label{R:last}
\-
\begin{enumerate}
\item
For $n=2$, the van der Geer--van der Vlugt family
$\{1+p^0,1+p^1,1+p^2,\ldots\}-\{p\}$ 
consists precisely of
all $p$-symmetric numbers with $s_p(-)=2$.
\item For $n=3$, the $p$-symmetric numbers with $s_p(-)=3$ 
include $1+p^{m}+p^{2m}$ for all $m\ge 1$ (see Example \ref{ex:p-symmetry}(9)).
\item
As $n$ grows, there are increasingly more $p$-symmetric numbers with $s_p(-)=n$, even though it becomes harder to exhibit them {\it all} besides Example \ref{ex:p-symmetry} 
(3) and (9). For example, for $n=4$, $(110011)_p$, $(11000011)_p$, and $(10100000101)_p$
are all $p$-symmetric numbers not in either of these families. 
\end{enumerate}
\end{remark}

\begin{acknowledgments}
We thank Gerard van der Geer for his encouragement and invaluable comments, and Daqing Wan for helpful feedback upon an early version.
Lastly, we are deeply grateful to the anonymous referees for their thorough 
reading, constructive comments and corrections, 
which have led to simpler and clearer proofs of Propositions \ref{P:determined}
and \ref{P:k-e}.
\end{acknowledgments}

\end{document}